\documentclass[12pt]{article}

\usepackage[english]{babel}
\usepackage{amssymb,amsmath}
\newtheorem{theorem}{Theorem }[section]
\newtheorem{lemma}[theorem]{Lemma}

\newtheorem{corollary}[theorem]{Corollary}

\newcommand{\proof}{\noindent\textbf{Proof. }}
\newcommand{\qed}{\hspace*{\fill}$\Box$}

\begin{document}

%In this version, we use the result from the paper ``On binary Kloosterman sums divisible by 3" written by Kseniya Garaschuk and Petr Lisonek.

\title{ Relative Difference Sets from Almost Perfect Nonlinear Functions}
\author{Zeying Wang\footnote{Department of Mathematics and Statistics, American University, Washington DC, USA; Email: zwang@american.edu} \\ \small{American University}}
\date{}
\maketitle

\begin{abstract}
In this paper we explore a connection between certain Almost Perfect Nonlinear Functions (APN functions) and relative difference sets. In particular, we show that the image set of certain 2-to-1 APN functions is a relative difference set. Through a result of Pott this further provides a connection between APN functions and bent functions.

\end{abstract}

{\bf Keywords:} APN function, relative difference set, bent function

\section{Introduction}

%$F : V_m\mapto V_m$ be a S-box.
 %The derivative of F with respect to $a \in V_m$
 % is
%the function DaF : Vm → Vm is defined as
%DaF (x) = F (x) + F (x + a), ∀x ∈ Vm.
%δ is an integer valued function from Vm × Vm is defined as
%δ(a, b) = |{x ∈ Vm, DaF (x) = b}| for a, b ∈ Vm.
%Abusing the notation δ, we define
%δ(F ) = max
%a̸=0,b∈Vm
%δ(a, b).
%δ(F ) needs to be as low as possible to resist diﬀerential attacks on block ciphers [7]. Since
%DaF (x) = DaF (x + a) for a ̸= 0 ∈ Vm, we have δ(F ) ≥ 2. The S-boxes for which the
%equality holds are the best choices against the diﬀerential attack.
%Definition 2 An S-box F : Vm → Vm is called Almost Perfect Nonlinear (in short, APN)
%if δ(F ) = 2.

{\bf Definition:} Let $p$ be a prime, $q=p^n$, and  $f$ be a map from $ \mathbb{F}_q$ to $ \mathbb{F}_q$. For any $a$, $b\in \mathbb{F}_q$, let
$$\delta_f(a,b):=|\{x \in \mathbb{F}_q\;:\: f(x+a)-f(x)=b\}|.$$
If $ \delta_f(a,b) \leq d$ for all $a \neq 0$ and $b\in \mathbb{F}_q$, then $f$ is called a {\it differentially $d$-uniform function} (abbreviated {\it $d$-uniform function}).  

\medskip

A 1-uniform map $f: \mathbb{F}_q\to\mathbb{F}_q$ is called {\it planar}, that is, $f$ is planar iff $f(x+a)-f(x)$ is a permutation for any $a \in \mathbb{F}_q^*$. Planar maps exists if and only if $q$ is odd.

A map $f$ is called {\it  almost perfect nonlinear (APN)} if $f$ is 2-uniform, that is,  if for every $a \neq 0$ and every $b \in \mathbb{F}$, the equation
$$f(x+a)-f(x)=b$$ 

admits at most 2 solutions.

\medskip

When $q$ is even, the equation $f(x+a)+f(x)=b$ has always an even number of solutions, since if $x$ is a solution, then obviously $x+a$ is also a solution.

\medskip

 Because there are no 1-uniform maps when $q$ is even, the APN maps have the smallest possible uniformity over binary fields.

APN functions over a finite field of characteristic 2 were introduced in \cite{Nyberg}  by K. Nyberg in 1993  and have been widely studied.

  APN maps and more
generally maps in characteristic 2 with low uniformity play an important role in
cryptography, mainly because they provide good resistance to so-called differential attacks when used
as an S(Substitution)-box of a block cipher.

\medskip

Below we mention some monomial APN functions over $\mathbb{F}_{2^n}$:

\begin{itemize}
\item Gold function: $x^d$, where $d=2^i+1$, and $\gcd(i, n)=1$, $1 \le i \le \frac{n-1}{2}$.

\item Kasami function: $x^d$, where $d=2^{2i}-2^i+1$, and $\gcd(i, n)=1$, $1 \le i \le \frac{n-1}{2}$.

\item Welch function: $x^d$, where $d=2^k+3$, and $n=2k+1$.

\end{itemize}

Some  important APN functions that are not monomials include

\begin{itemize}
\item the function $x^3+Tr(x^9)$  over $\mathbb{F}_{2^n}$ for any $n$.

\item  the function $x^{2^i+1}+(x+x^{2^m})^{2^i+1}$ over $\mathbb{F}_{2^n}$ when $n=2m$, $m$ is even and $\gcd(n,i)=1$.
\end{itemize}

In \cite{Kolsch}, K\"olsch et al. considered image sets of differentially $d$-uniform maps over finite fields and showed that APN maps with minimal image size are very close to being 3-to-1. In particularly,  they either mentioned or explicitly proved that certain families of APN functions over $\mathbb{F}_{2^n}$, with $n$ odd, are 2-to-1 functions.  Those 2-to-1 APN functions are of particular interest to us, because, in combination with a prior result of the author of the current paper (see \cite{WQWX}) related to 2-to-1 planar functions, they inspired us to investigate the image set of these 2-to-1 APN functions. We cite this result from  \cite{WQWX} below:

\begin{theorem}\label{WQWX} {\cite{WQWX}}

Let $G$ and $H$ be two finite groups of the same order $v$. Let $f: G\to H$ be a 2-to-1 planar function and $D=f(G)\setminus\{1_H\}$. 
\begin{itemize}
\item[(1)] if $v\equiv 3 \pmod 4$, then $D$ is a skew Hadamard difference set in $H$.
\item[(2)] If $v\equiv 1 \pmod 4$, then $D$ is a $(v, (v-1)/2, (v-5)/4, (v-1)/4)$ partial difference set in $H$.
 \end{itemize}
 \end{theorem}

Given this result and the existence of 2-to-1 APN functions as shown in \cite{Kolsch} we believed it was now natural to wonder if 2-to-1 APN functions could have image sets with interesting properties. Before investigating these image sets in the next section we first provide some further definitions.
\medskip

{\bf Definition:} Let $G$ be a finite  group of order $mn$ containing a normal subgroup $N$ of order $n$. A $k$-subset $R$ of $G$ is called a {\it relative ($m$, $n$, $k$, $\lambda$) difference set} (relative to $N$) if the list of differences $r-r'$  with $r$, $r'\in R$ and $r\neq r'$ covers every element in $G\setminus{N}$ exactly $\lambda$ times and no element in $N\setminus\{0\}$ is covered. The subgroup $N$ is called the ``{\it forbidden subgroup}".\\

The differential uniformity, and thus the APN property, is preserved by some transformations of functions, which define equivalence relations between vectorial Boolean functions. Two of these equivalence notions are the extended affine equivalence (EA-equivalence) and Carlet-Charpin-Zinoviev equivalence (CCZ-equivalence). EA-equivalence is a particular case of CCZ-equivalence, which is the most general known equivalence relation preserving the differential uniformity

\medskip

{\bf Definition:} Two vectorial Boolean functions $F$  and $G$ over $\mathbb{F}_{2^n}$ are {\it Extended Affine (EA) equivalent} iff there exist two affine permutations $A$,  $B$, and an affine function $C$ such that $G=A\circ F\circ B+C$.

\medskip

As we will not need the general notion of CCZ-equivalence we will restrict ourselves to the special case for power functions that we will use in the next section.
\medskip

{\bf Definition:} For $k$ coprime with $2^n-1$ let $x^k$ and $x^l$ be two APN functions over $\mathbb{F}_{2^n}$.  The two APN functions  $x^k$ and $x^l$ over $\mathbb{F}_{2^n}$ are {\it CCZ-equivalent} iff there exists an integer $0\le a <n$ such
that $l \equiv k 2^a \pmod {2^{n}-1}$ or $kl\equiv 2^a \pmod {2^n-1}$

\medskip

In the above case CCZ-equivalence is actually the same as EA-equivalence and also the same as what is known as cyclotomic equivalence. We refer the reader to both \cite{Dempwolff} and \cite{Yoshiara} for more details.
%\noindent{\bf Corollary} (Pott, Nonlinear function and RDS 2004)

%A function $f$: $K \rightarrow N$ with elementary abelian $2$-group $K$ and $N$ of size $2^{2t+1}$ each is almost perfect nonlinear if it is maximal nonlinear. (The converse is not true.)

%In the next section we will show that the image set of three families of 2-to-1 APN functions from $\mathbb{F}_{2^n}$ to $\mathbb{F}_{2^n}$ are relative difference sets in $(\mathbb{F}_{2^{n}},+)$.

\section{Main Results}

In this section we will prove that the image set of three families of 2-to-1 APN functions from $\mathbb{F}_{2^n}$ to $\mathbb{F}_{2^n}$ are relative difference sets in $(\mathbb{F}_{2^{n}},+)$.

\subsection{The APN map $f(x)= x^3+a^{-1}\mathrm{Tr}(a^3x^9)$ }

In this subsection we will show that two 2-to-1 APN functions related to the map $x \to x+a^{-1}\mathrm{Tr}(a^3x^3)$ give rise to relative difference sets. We will need the following result from \cite{Kolsch}:

\begin{lemma}\label{image_size}
Let  $a$ be a non-zero element in $\mathbb{F}_{2^n}$. If $n$ is odd, then the map $x \to x+a^{-1}\mathrm{Tr}(a^3x^3)$ is a 2-to-1 function. 
\end{lemma}

\begin{theorem}\label{main}
Let $n=2k+1$ be odd and $a \in \mathbb{F}_{2^n}^*$. Then the image set of the function $f(x)= x+a^{-1}\mathrm{Tr}(a^3x^3)$ is a {\rm (}$2^{2k}$,  $2$,  $2^{2k}$,  $2^{2k-1}${\rm )} relative difference set in $(\mathbb{F}_{2^{n}},+)$ relative to the subgroup $N=\{a^{-1},0\}$
\end{theorem}

\proof Let $D$ be the image set of $f(x)$. That is,
$$D=\{x+a^{-1}\mathrm{Tr}(a^3x^3), x \in \mathbb{F}_{2^n}\}.$$

We first show that $a^{-1}$ cannot be represented by any of the  differences $x_1-y_1$, where $x_1$, $y_1\in D$.  Assume by way of contradiction that $a^{-1}=x+a^{-1}\mathrm{Tr}(a^3x^3)-(y+a^{-1}\mathrm{Tr}(a^3y^3))$, where $x\neq y$ and $x$, $y\in \mathbb{F}_{2^n}$. By combining similar terms, we get
\begin{eqnarray}
x-y=-a^{-1}\left(\mathrm {Tr}(a^3x^3-a^3y^3)\right)+a^{-1}. \label{eq_general}
\end{eqnarray}
\begin{itemize}
\item[Case 1:] If $\mathrm{Tr}(a^3x^3-a^3y^3)=0$, then by  Equation (\ref{eq_general}) we have $x-y=a^{-1}$. It follows that $x=a^{-1}+y$, thus $ax=1+ay$. Substituting $1+ay$ for $ax$ in $\mathrm{Tr}(a^3x^3-a^3y^3)=0$ gives
$$ 0=\mathrm{Tr}((1+ay)^3-a^3y^3)=\mathrm{Tr}(1+ay+a^2y^2)=\mathrm{Tr}(1)=1,$$
a contradiction. Here we used that $\mathrm{Tr}(ay)+\mathrm{Tr}(a^2y^2)=0$ and that $\mathrm{Tr}(1)=1$ because $n$ is odd.

%\begin{eqnarray}
%\nonumber\mathrm{Tr}((1+ay)^3-a^3y^3)&=&0\\
%\nonumber\mathrm{Tr}(1+ay+a^2y^2)&=&0\\
%\mathrm{Tr}(ay+a^2y^2)&=&-\mathrm{Tr}(1)=1.\label{trace_eq1} 
%\end{eqnarray}
%But since $a^2y^2=(ay)^2$, by the property of trace function, $\mathrm{Tr}(a^2y^2)=\mathrm{Tr}(ay)$, thus $\mathrm{Tr}(ay+a^2y^2)=0$, contradicting with Equation (\ref{trace_eq1}).

\item[Case 2:] If $\mathrm{Tr}(a^3x^3-a^3y^3)=1$, then by Equation (\ref{eq_general}) we have $x-y=0$, contradicting with the assumption that $x \neq y$. 
\end{itemize}

\bigskip

Next let $b$ be a nonzero element in $\mathbb{F}_{2^n}$ with $b\neq a^{-1}$. We count the number of pairs $(x_1,y_1)$ with $x_1$,  $y_1\in D$ such that $b=x_1-y_1$. \\

 We  first count the number of  pairs $(x,y)$ that satisfy the equation
\begin{eqnarray}
x+a^{-1}\mathrm{Tr}(a^3x^3)-(y+a^{-1}\mathrm{Tr}(a^3y^3))=b.\label{difference_eq1}
\end{eqnarray}

%Next we will show that if $b \neq a^{-1}$, then Equation (\ref{difference_eq1}) has exactly $2^{n}$ solutions. 

Clearly Equation (\ref{difference_eq1}) is equivalent to 
\begin{eqnarray}
(x-y)+a^{-1}(\mathrm{Tr}(a^3(x^3-y^3)))&=&b \label{difference_eq2}\\
\mbox{and hence } \quad\quad a(x-y)+\mathrm{Tr}(a^3(x^3-y^3))&=&ab \label{difference_eq3}
\end{eqnarray}
Similar to what we have done for $b=a^{-1}$, we will discuss the number of solutions of Equation (\ref{difference_eq3}) in two cases based on the values of the trace function:
\begin{itemize}
\item[Case 1:] If $\mathrm{Tr}(a^3(x^3-y^3))=0$, then from Equation (\ref{difference_eq3}) we have $a(x-y)=ab$, that is, $ax=ay+ab$.  Let  $c=ab$ and $z=ay$. Clearly $c \neq 0$ and $c\neq 1$ because $b\neq a^{-1}$. Then $ax=z+c$ and  
\begin{eqnarray}
\nonumber a^3x^3-a^3y^3=(z+c)^3-z^3=cz^2+c^2z+c^3
\end{eqnarray}
and $\mathrm{Tr}(a^3(x^3-y^3))=0$ becomes
\begin{equation}
\mathrm{Tr}(cz^2+c^2z+c^3)=0\label{general_eq1}
\end{equation}
Since $\mathrm{Tr}(c^2z)=\mathrm{Tr}(c^4z^2)$, Equation (\ref{general_eq1}) is equivalent to 
\begin{equation}
\mathrm{Tr}((c^4+c)z^2+c^3)=0.\label{derived_eq1}
\end{equation} 
Recall there are exactly $2^{n-1}$ solutions to $\mathrm{Tr}(t)=0$. If $r$ is a solution to $\mathrm{Tr}(t)=0$, then 
\begin{eqnarray*}
(c^4+c)z^2+c^3&=&r\\
z^2&=&(r-c^3)(c^4+c)^{-1} \quad \mbox{if $c^4+c\neq 0$ }\\
z&=&(r-c^3)^{2^{n-1}}(c^4+c)^{-2^{n-1}}.
\end{eqnarray*}
Since $c \neq 0, 1$ and $\gcd (3, 2^n-1)=1$, we have $c^4+c\neq 0$.  Because $\gcd(2^{n-1}, 2^n-1)=1$, for each solution $r$, there is a unique solution $z$ of Equation (\ref{derived_eq1}) corresponding to it.  Also for each solution $z$, there is a unique pair of solution $(x,y)$ of Equation (\ref{difference_eq1}) corresponding to it ($y=a^{-1}z$ and $x=y+b$).

Hence  there are exactly $2^{n-1}$ pairs of  solutions $(x,y)$  satisfying Equation (\ref{difference_eq1}) and $\mathrm{Tr}(a^3(x^3-y^3))=0$.

\item[Case 2:] If $\mathrm{Tr}(a^3(x^3-y^3))=1$, then from Equation (\ref{difference_eq3}) we have $a(x-y)=ab-1$, that is, $ax=ay+(ab-1)$.  Let  $d=ab-1$ and $z=ay$. Clearly $d\neq 0$ since $b\neq a^{-1}$,  and $d\neq 1$ since both $a$ and $b$ are nonzero elements. Then $ax=z+d$ and  
\begin{eqnarray}
\nonumber a^3x^3-a^3y^3=(z+d)^3-z^3=dz^2+d^2z+d^3
\end{eqnarray}
and $\mathrm{Tr}(a^3(x^3-y^3))=1$ becomes
\begin{equation}
\mathrm{Tr}(dz^2+d^2z+d^3)=1\label{general_eq2}
\end{equation}
Using similar argument as in Case 1, we can show that there are exactly $2^{n-1}$ pairs of solutions to Equation (\ref{difference_eq1}) and $\mathrm{Tr}(a^3(x^3-y^3))=1$.\\

\end{itemize}

Combining the results of Case 1 and 2, there are $2^{n-1}+2^{n-1}=2^n$ solutions to Equation (\ref{difference_eq1}).  By Lemma \ref{image_size}, there are $2^{n}/4=2^{2k+1-2}=2^{2k-1}$ different representations of $b$ with $b=x_1-y_1$, $x_1\neq y_1$ and $x_1,y_1 \in D$ whenever $b\neq a^{-1}$. 
\medskip

In conclusion, this shows that the set $D$ as defined at the beginning of the proof is a {\rm (}$2^{2k}$,  $2$,  $2^{2k}$,  $2^{2k-1}${\rm )} relative difference set in $(\mathbb{F}_{2^{n}},+)$ relative to the subgroup $N=\{a^{-1},0\}$.

\qed

\medskip

For any permutation function $g(x)$ of the finite field $\mathbb{F}_{2^n}$, $(f\circ g)(x)=f(g(x))$  has the same image set as $f(x)$. Thus we have the following corollary:

\begin{corollary}\label{permutation_general}
Let $n=2k+1$ be odd and $a \in \mathbb{F}_{2^n}$ non-zero, and $g(x)$ is a permutation function of $\mathbb{F}_{2^n}$. Then the function $g(x)+a^{-1}\mathrm{Tr}(a^3 (g(x)^3)$ is a 2-to-1 function and the the image of the function  $g(x)+a^{-1}\mathrm{Tr}(a^3 (g(x)^3)$ is a {\rm (}$2^{2k}$,  $2$,  $2^{2k}$,  $2^{2k-1}${\rm )} relative difference set in $(\mathbb{F}_{2^{n}},+)$ relative to the subgroup $N=\{a^{-1},0\}$.
\end{corollary}

When $n$ is odd, we have $\gcd(3,2^n-1)=1$,  and the monomial $x^3$ is a permutation function of $\mathbb{F}_{2^n}$.  As a result, the following result  follows from Corollary \ref{permutation_general} :

\begin{theorem}\label{APN_IEEE}
Let $n=2k+1$ be odd and $a \in \mathbb{F}_{2^n}$ non-zero. Then the image of the APN map $x \to x^3+a^{-1}\mathrm{Tr}(a^3x^9)$ is a {\rm (}$2^{2k}$,  $2$,  $2^{2k}$,  $2^{2k-1}${\rm )} relative difference set in $(\mathbb{F}_{2^{n}},+)$ relative to the subgroup $N=\{a^{-1},0\}$.
\end{theorem}

Note that the authors of \cite{Carlet} proved that the above mentioned map $x \to x^3+a^{-1}\mathrm{Tr}(a^3x^9)$  is indeed an APN map.
\medskip

Our results so far can be used to show another 2-to-1 APN function also gives rise to a relative difference set.
\medskip

\noindent {\bf Definition:} We call a function $l(x)$ over $\mathbb{F}_{2^n}$   a {\it linear function over $\mathbb{F}_2$} if $l(x-y)=l(x)-l(y)$ for any $x$, $y\in \mathbb{F}_{2^n}$.  
\medskip

We immediately obtain the following result:

\begin{corollary}

Let $n=2k+1$ be odd and $a \in \mathbb{F}_{2^n}$ non-zero,  $h(x)= x^3+a^{-1}\mathrm{Tr}(a^3x^9)$, and $l(x)$ is a linear function over $\mathbb{F}_{2}$. Then the image of the composite function  $(l\circ h)(x)$ is a {\rm (}$2^{2k}$,  $2$,  $2^{2k}$,  $2^{2k-1}${\rm )} relative difference set in $(\mathbb{F}_{2^{n}},+)$ relative to the subgroup $N=\{l(a^{-1}),0\}$.
\end{corollary}

 The authors of \cite{Kolsch} showed the following family of  functions is also APN and  2-to-1:
 
 \begin{lemma}\label{APN_gamma}
 Let $\alpha$, $\beta$, $\gamma$ be non-zero elements in $\mathbb{F}_2^n$ with $n$ odd. Further, let $\gamma \notin \{x^2+\alpha x\,|\, x\ in \mathbb{F}_{2^n}\}$ and $Tr(\beta \alpha)=1$, then
 $$k(x)=x^6+\alpha x^3+\gamma Tr (\alpha^{-3}x^9+\beta x^3)$$
 is APN and 2-to-1.
 \end{lemma}
 
 As stated in  \cite{Kolsch}, the APN function from Lemma \ref{APN_gamma} is very closely related to the APN function in Theorem \ref{APN_IEEE}: \\
 
 Let $l(x)=x^2+\alpha x+\gamma Tr(\beta x)$ and $h(x)=x^3+\alpha\mathrm{Tr}(\alpha^{-3}x^9)$, where $\gamma \notin \{x^2+\alpha x\,|\, x\in \mathbb{F}_{2^n}\}$ and $Tr(\beta \alpha)=1$.  It is easy to see that $l(x)$ is a linear function over $\mathbb{F}_2$ and   $l(\alpha)=\alpha^2+\alpha^2+\gamma Tr(\beta \alpha)=\gamma$.
Also in \cite{Kolsch}, it was shown that $l(x)$ is a permutation function over $\mathbb{F}_{2^n}$. We compute
 \begin{eqnarray}
 \nonumber(l\circ h)(x)=l(h(x))&=&l(x^3+\alpha\mathrm{Tr}(\alpha^{-3}x^9))\\
\nonumber &=&l(x^3)+l(\alpha \mathrm{Tr}(\alpha^{-3} x^9)\\
\nonumber &=& x^6+\alpha x^3+\gamma Tr(\beta x^3)+\mathrm{Tr}(\alpha^{-3} x^9) l(\alpha)\\
 &=&x ^6+\alpha x^3+\gamma Tr(\beta x^3+\alpha^{-3} x^9) \label{explanation}
 \end{eqnarray}
 In Equation (\ref{explanation}), we used $l(\alpha)=\gamma$.   Based on Lemma \ref{APN_gamma}, we have the following result:
 
 \begin{theorem}
 Let $\alpha$, $\beta$, $\gamma$ be non-zero elements in $\mathbb{F}_2^n$ with $n$ odd. Further, let $\gamma \notin \{x^2+\alpha x\,|\, x\in \mathbb{F}_{2^n}\}$ and $Tr(\beta \alpha)=1$, then the image of the 2-to-1 APN function
 $$k(x)=x^6+\alpha x^3+\gamma Tr (\alpha^{-3}x^9+\beta x^3)$$
 is a {\rm (}$2^{2k}$,  $2$,  $2^{2k}$,  $2^{2k-1}${\rm )} relative difference set in $(\mathbb{F}_{2^{n}},+)$ relative to the subgroup $N=\{l(\alpha),0\}=\{\gamma, 0\}$.
\end{theorem}

Note that further variations on these functions will provide relative difference sets as the following result is obvious:

\begin{corollary}
Let $g(x)$ is a permutation function of $\mathbb{F}_{2^n}$ and $l(x)$ be a linear permutation function over $\mathbb{F}_{2^n}$. If the image of a function $f$ over $\mathbb{F}_{2^n}$ is a $(m, n, k, \lambda)$ relative difference set relative to the subgroup $N$, then the image of $(l\circ f\circ g)(x)=l(f(g))(x)$ is a $(m,n,k,\lambda)$ relative difference set relative to the subgroup of $l(N)$. 
\end{corollary}

\subsection{The APN map $f(x)=x^{2^k-1}+x^{2^k}$}

Next we turn our attention to the family of functions $f(x)=x^{2^k-1}+x^{2^k}$, $k\ge 2$.  They are 2-to-1 APN functions over $(\mathbb{F}_{2^{2k-1}},+)$ and we will show in what follows that their images are  relative difference sets in  $(\mathbb{F}_{2^{2k-1}},+)$.
\medskip

Below we cite a result from \cite{Garaschuk} (Theorem 2.2) stating some important properties of $f(x)=x^{2^k-1}+x^{2^k}$.

\begin{lemma}\label{G-L-Theorem}

Let $m>1$ and let $k$ be such that $\gcd(2^k-1, 2^m-1)=1$. Then for each $a \in \mathbb{F}_{2^m}$ we have $Tr(a^{1/(2^k-1)})=0$ if and only if $a=t^{2^k}+t^{2^k-1}$ for some $t \in \mathbb{F}_{2^m}$. Furthermore, the function $f(t)=t^{2^k}+t^{2^k-1}$ is a 2-to-1 function on $\mathbb{F}_{2^m}$.
\end{lemma}

{\it Note:} The latter part of Theorem \ref{G-L-Theorem} does not appear explicitly in the statement of the Theorem 2.2 in \cite{Garaschuk}, but appears in their proof of the theorem. 

\medskip

We are now ready to show another family of 2-to-1 APN functions give rise to relative different ses.

\begin{theorem}\label{Special_family}
Let $f(x)=x^{2^k-1}+x^{2^k}$, $k\ge 2$. Then $f(x)$ is a 2-to-1 APN function in $\mathbb{F}_{2^{2k-1}}$, and its image is a $(2^{2k-2}, 2, 2^{2k-2}, 2^{2k-3})$ relative difference set in $(\mathbb{F}_{2^{2k-1}},+)$ relative to the subgroup $N=\{0, 1\}$
\end{theorem}

\proof Let $D$ be the image set of $f(x)$. That is,
$$D=\{x^{2^k-1}+x^{2^k}, x \in \mathbb{F}_{2^{2k-1}}\}.$$

Since $(2^k+1)(2^k-1)-2(2^{2k-1}-1)=1$, it follows that $\gcd(2^k-1, 2^{2k-1}-1)=1$,  and $(2^k+1)(2^k-1)\equiv 1 \pmod {2^{2k-1}-1}$, thus $x^{2^k-1}$ is CCZ-equivalent to the Gold function $x^{2^k+1}$ (an APN function since $\gcd(k, 2k-1)=1$, thus $x^{2^k-1}$ is an APN function. Since $x^{2^k}$ is a linear function, $x^{2^k-1}+x^{2^k}$ is EA equivalent to $x^{2^k-1}$, thus $x^{2^k-1}+x^{2^k}$ is an APN function.  By Lemma \ref{G-L-Theorem}, the function $f(x)$ is a 2-to-1 function, thus $|D|= 2^{2k-2}$.
\medskip

Let $(x^{2^k-1}+x^{2^k})+(y^{2^k-1}+y^{2^k})=b$, and let $u=x^{2^k-1}+x^{2^k}$, and $v=y^{2^k-1}+y^{2^k}$. Clearly
$(u^{2^k+1})^{2^k-1}=u$, thus $u^{1/(2^k-1)}=u^{2^k+1}$. (Note that every element of $\mathbb{F}_{2^{2k-1}}$ has a unique $(2^k-1)$-th root, so this is well-defined.) By  Lemma \ref{G-L-Theorem},  the equations $u=x^{2^k-1}+x^{2^k}$, and $v=y^{2^k-1}+y^{2^k}$ are solvable if and only if 
 $Tr(u^{1/(2^k-1)})$=$Tr(u^{2^k+1})=0$ and $Tr(v^{1/(2^k-1)})$=$Tr(v^{2^k+1})=0$. 
 
 \smallskip

 We first handle the case $b=1$. 
  If $u+v=1$, it follows that $v=1+u$, and 
\begin{eqnarray}
Tr(v^{2^k+1})&=&Tr((1+u)^{2^k+1})=Tr((1+u)^{2^k}(1+u))\nonumber\\
                           &=&Tr((1+u^{2^k})(1+u)=Tr(1+u+u^{2^k}+u^{2^k+1})\nonumber\\
                           &=&Tr(1)+Tr(u+u^{2^k})+Tr(u^{2^k+1})=1\nonumber                     
\end{eqnarray}

contradicting with $Tr(v^{2^k+1})=0$. Here we used that $Tr(u+u^{2^k})=0$, $Tr(u^{2^k+1})=0$, and $Tr(1)=1$ in $\mathbb{F}_{2^{2k-1}}$.

Thus the identity element 1 cannot be written as a difference of $u$, $v$, where $u$, $v$ are from the image set $D$.

\bigskip

Next let $b$ be a nonzero element in $\mathbb{F}_{2^{2k-1}}$, and $b \neq 1$. Assume $u+v=b$ and $Tr(u^{2^k+1})=Tr(v^{2^k+1})=0$. We will count the number of pairs  $(u, \, v)$  satisfying those conditions.
 First we will rewrite $v$ as $u+b$. We see that
\begin{eqnarray}\label{trace_combination}
Tr(v^{2^k+1})=Tr((u+b)^{2^k+1})=Tr(u^{2^k+1}+u^{2^k}b+ub^{2^k}+b^{2^k+1})=0.
\end{eqnarray}
Combining the condition $Tr(u^{2^k+1})=0$ and the identity 
$Tr(u^{2^k} b)$=$Tr(   (u^{2^k}b)^{2^{k-1}}   )$\
=$Tr(ub^{2^{k-1}})$, Equation (\ref{trace_combination}) becomes
\begin{eqnarray}\label{trace_simple_form}
Tr((b^{2^{k-1}}+b^{2^k})u)+Tr(b^{2^k+1})=0
\end{eqnarray}

\smallskip
Let $U_b=\{u \,|\, Tr((b^{2^{k-1}}+b^{2^k})u)=0\}$. Clearly $U_b$ is a hyperplane  passing through $1$ if $b\neq 0, 1$. Let $T=\{u\,|\, Tr(u^{2^k+1})=0\}$ and  $i_b=|T \cap U_b|$. Since $gcd(2^{2k-1}-1,2^k+1)=1$, we can see that $u^{2^k+1}$ is a permutation function in $\mathbb{F}_{2^{2k-1}}$, thus  $|T|=2^{2k-2}$. Next we will discuss the number of solutions to the system of equations $u+v=b$ and $Tr(u^{2^k+1})=Tr(v^{2^k+1})=0$ in two cases.

\begin{itemize}
\item[(i)]
We first consider the case when $Tr(b^{2^k+1})=0$, that is , when $Tr((b^{2^{k-1}}+b^{2^k})u)=0$ by using Equation (\ref{trace_simple_form}).

  Since $U_b$ is a hyperplane and $1 \in U_b$, it is easy to see that  $u \in U_b$ if and only if $u+1 \in U_b$. On the other hand, one and only one of $u$ and $u+1$ must be in  $T$ since $Tr((u+1)^{2^k+1})$=$Tr(u^{2^k+1}+u^{2^k}+u+1)$
  =$Tr(u^{2^k+1})+Tr(1)$=$Tr(u^{2^k+1})+1$.  Since $|U_b|=2^{2k-2}$, and exactly half of the elements of $U_b$  are in $T$, we see that
$i_b = |U_b|/2=2^{2k-3}$ for all $b \neq 0, 1$. Note that $i_b$ is also exactly half of the size of $T$. That is, there are exactly $2^{2k-3}$ pairs of  $(u, v)$ such that  $Tr(u^{2^k+1})=Tr(v^{2^k+1})=0$, and $Tr((b^{2^{k-1}}+b^{2^k})u)=0$.

\item[(ii)] Next we will consider the case when $Tr(b^{2^k+1})=1$, that is, when $Tr((b^{2^{k-1}}+b^{2^k})u)=1$. As shown in Case (i), exactly half of the elements of $U_b$  are in $T$.  Since the set of the solutions of $Tr((b^{2^{k-1}}+b^{2^k})u)=1$ is  the complement of the hyperplane $U_b$ with the same size $2^{2k-2}$, it follows that half of the elements of $T$ are from $U_b$, and half of the elements of $T$ are from the complement of $U_b$.  Thus there are also exactly $2^{2k-3}$ pairs of  $(u, v)$ such that  $Tr(u^{2^k+1})=Tr(v^{2^k+1})=0$, and $Tr((b^{2^{k-1}}+b^{2^k})u)=1$.

Hence in both cases, we have shown that there are exactly $2^{2k-3}$ pairs of  $(u, v)$ such that  $u+v=b$, where $u=x^{2^k-1}+x^{2^k}$, and $v=y^{2^k-1}+y^{2^k}$, and $b\neq 0, 1$. This completes the proof.

\end{itemize} \qed

\bigskip

 {\bf Note:} It is known that $f(x)=x^3+x^4$ is a 2-to-1 APN function in $\mathbb{F}_{2^n}$ for all odd $n$ is odd. However, computations show that the image set of $f(x)$ is not always a relative difference set. Hence not every 2-to-1 APN function gives rise to a relative difference set!

\section{On bent functions derived from the image of some APN functions}\label{Bent}

In this short section we briefly explore how a result of A. Pott relates our results on APN functions to bent Boolean functions. For completeness we provide one of several equivalent definitions of a bent function below (see for example Section 6.1 of \cite{Carlet_book}):
\medskip

{\bf Definition:}  Let $n$ be even. An $n$-variable Boolean function $f: \mathbb{F}_{2^n} \to \mathbb{F}_2$ is {\it bent} if and only if its Hamming distance to any affine function equals $2^{n-1} \pm 2^{\frac{n}{2}-1}$.
\medskip

Extended Affine equivalence for bent functions is defined in a similar way as for APN functions. Bent functions are of great interest because they play an important role in cryptography and coding theory, see for example \cite{Carlet_book}. In 2004, A. Pott \cite{Pott} proved the following theorem:

\begin{theorem}\label{Pott}
Let $K$ and $N$ be arbitrary finite groups and $f:\; K\to N$, the set 
$$R:=\{(g, f(g)):\; g \in K\} \subset K \times N$$
is a semiregular splitting $(|K|, |N|, |K|, |K|/|N|)$-difference set in $K\times N$ relative to $\{1\}\times N$ if and only if $f$ is perfect nonlinear.
\end{theorem}

Whenever $K=\mathbb{F}_{2^n}$ and $N=\mathbb{F}_2$ in the previous theorem this means the function $f$ is bent. This case corresponds exactly to the relative difference sets we discussed in Section 2 of this paper.
\medskip

In Theorem \ref{main}, we showed that the image $D$ of the APN function $h(x)=x+a^{-1}\mathrm{Tr}(a^3x^3)$ is a ($2^{2k}$, 2, $2^{2k}$, $2^{2k-1}$)-relative difference set in $(\mathbb{F}_{2^{2k+1}}, +)$ relative to the subgroup $N=\{a^{-1},0\}$, which is a semiregular splitting difference set satisfying the condition in the above result. In the next theorem we describe the bent function arising from it (which, unfortunately, turns out to be known).

\begin{theorem}
Let $n=2k+1$. The bent function arising from the APN function $h(x)=x+a^{-1}\mathrm{Tr}(a^3x^3)$ over $\mathbb{F}_{2^n}$  through Theorem \ref{APN_IEEE} and Theorem \ref{Pott} is affine equivalent to the bent function $x_1x_2\oplus x_3x_4 \oplus \cdots \oplus x_{n-2}x_{n-1}\oplus \epsilon $, $\epsilon \in \mathbb{F}_2$.
\end{theorem}

\proof Let $\{a^{-1}, \omega_1, \omega_2, \cdots, \omega_{n-1}\}$ be a basis of the vector space $\mathbb{F}_{2^n}$ over $\mathbb{F}_2$. Then we can write $x$ as 
$x=x_0 a^{-1}+x_1 \omega_1+x_2\omega_2+\cdots+x_{n-1}\omega_{n-1}$ where $x_i \in \mathbb{F}_2$, and
 $$
\{h(x):x\in\mathbb{F}_{2^n}\}= \{F_h(x_1,\dots,x_{n-1}),x_1,x_2,\dots,x_{n-1}):x_1,\dots,x_{n-1}\in \mathbb{F}_2\}
$$
 
The function  $F_h$ defined by $h$  is
$$
F_h(x_1,\dots,x_n)=x_0+\text{Tr}(a^3 x^3)$$
%((\sum_{i=0}^{n}x_i\omega^i)^3)=\text{Tr}((\sum_{i=1}^{n}x_i\omega^i)^3)

  Note however that  $F_h$ does not depend on the value of $x_0$, because $h(x)=h(y)$ if and only if $x=y$ or $x+y=a^{-1}$. As a result from now on we only need to consider those $x$ where $x_0=0$, that is,  $x=x_1 \omega_1+x_2\omega_2+\cdots+x_{n-1}\omega_{n-1}$,  $x_i \in \mathbb{F}_2$.\\

Since the image $D$ of $h$ is  a ($2^{2k}$, 2, $2^{2k}$, $2^{2k-1}$)-relative difference set in $(\mathbb{F}_{2^{2k+1}}, +)$ relative to the subgroup $N=\{a^{-1},0\}$, by Theorem \ref{Pott}, the function $F_h$ defined by $h$ is a bent function. Next we want to show that $F_h(x_1, x_2, \cdots, x_{n-1})$ is quadratic function in terms of $x_1$, $x_2$, $\cdots$, $x_{n-1}$.\\

Since $x=x_1 \omega_1+x_2\omega_2+\cdots+x_{n-1}\omega_{n-1}$, we have $ax=x_1 (a \omega_1)+x_2 (a \omega_2)+\cdots+x_{n-1}(a \omega_{n-1})$. For simplicity of writing, we denote $b_i=a \omega_i$, $i=1, 2, \cdots, n-1$, so $ax=x_1 b_1+x_2 b_2+\cdots+x_{n-1} b_{n-1}.$ We compute, using $x_i^2=x_i$ as $x_i\in\mathbb{F}_2$,

\begin{eqnarray}
&&\mathbb{F}_h(x_1,x_2,\cdots, x_{n-1})=x_0+\text{Tr}(a^3 x^3)=\text{Tr}(a^3 x^3) \nonumber \\
&&=\text{Tr}((x_1 b_1+x_2 b_2+\cdots x_{n-1}b_{n-1})^3)\nonumber \\ 
&&=\text{Tr}\left((x_1 b_1+x_2 b_2+\cdots x_{n-1}b_{n-1})^2 (x_1 b_1+x_2 b_2+\cdots x_{n-1}b_{n-1})\right)\nonumber \\ 
&&=\text{Tr}\left((x_1 b_1^2+x_2 b_2^2+\cdots x_{n-1}b_{n-1}^2)(x_1 b_1+x_2 b_2+\cdots x_{n-1}b_{n-1})\right) \nonumber\\  
&&=\text{Tr}(\sum_{i=1}^{n-1}\sum_{j=1}^{n-1} b_i^2 b_j x_i x_j)=\sum_{i,j}a_{ij}x_i x_j \text{\hspace{1cm} where $a_{ij}=\text{Tr}(b_i^2 b_j)$. }\nonumber
\end{eqnarray}

Hence we have shown that $F_h(x)$ is a Boolean quadratic bent function. By \cite{Carlet} all Boolean quadratic bent function are known and affine equivalent to the function $x_1x_2\oplus x_3x_4 \oplus \cdots \oplus x_{n-1}x_n\oplus \epsilon $,  $\epsilon \in \mathbb{F}_2$. This concludes the proof. \qed

\bigskip

In the case of the APN function $x^{2^k-1}+x^{2^k}$ discussed in Section 2.2 we were currently unable to determine the corresponding bent function. Computations in small cases however seem to indicate here as well the corresponding bent function is quadratic.

\section{Open problems}

We conclude the paper with some natural open problems.

Given that clearly not every 2-to-1 APN function gives rise to a relative difference set it is natural to ask the following:

\begin{itemize}
\item Which APN functions do give rise to relative difference sets? What characterizes these APN functions?
\item Is it true that every APN function is EA or CCZ equivalent to an APN function that gives rise to a relative difference set?
\end{itemize}

Further, given the natural correspondence with bent functions as described in the previous section, it is natural to ask

\begin{itemize}
\item Does every bent function arise from some 2-to-1 APN function through the connection described in Section 3? And if not, which do?
\end{itemize}

\section*{Acknowledgement}

The author would like to thank Yue Zhou for computational work done during the 7th Irsee conference that provided early evidence for some of the results from this paper, as well as for providing the blueprint for the proof in Section \ref{Bent}.

\end{document}